\documentclass[a4paper,12pt, reqno]{amsart}
\usepackage{amsthm, amsmath, amssymb, amsfonts, amsopn, color}
\usepackage{enumerate}
 \usepackage{amsfonts,amssymb}

\newtheorem{theorem}{Theorem}[section]

\newtheorem{lemma}[theorem]{Lemma}
\numberwithin{equation}{section}
\begin{document}
\title[Uniqueness of topological solutions]{Uniqueness of topological multi-vortex solutions for a skew-symmetric Chern-Simons system}

\author{Hsin-Yuan Huang}
\address[Hsin-Yuan Huang]{Department of Applied Mathematics, National Sun Yat-sen Universtiy, Kaoshiung, Taiwan}
\email{hyhuang@math.nsysu.edu.tw}

\author{Youngae Lee}
\address[Youngae Lee] {Center for Advanced Study in Theoretical Sciences,
National Taiwan University,
No.1, Sec. 4, Roosevelt Road, Taipei 106, Taiwan}
\email{youngaelee0531@gmail.com}

\author{Chang-Shou Lin}
\address[Chang-Shou Lin] {Taida Institute for Mathematical Sciences,
Center for Advanced Study in Theoretical Sciences,
National Taiwan University,
No.1, Sec. 4, Roosevelt Road, Taipei 106, Taiwan}
\email{cslin@tims.ntu.edu.tw}

\begin{abstract}

Consider  the following  skew-symmetric Chern-Simons system
\begin{equation*}\left \{
\begin{split}
&\Delta u_{1}+\frac{1}{\varepsilon^2} e^{u_{2}}(1-e^{u_{1}})=4\pi \sum^{N_1}_{j=1}\delta_{p_{j,1}}\\
&\Delta u_{2}+\frac{1}{\varepsilon^2} e^{u_{1}}(1-e^{u_{2}})=4\pi \sum^{N_2}_{j=1}\delta_{p_{j,2}}
\end{split}\right.\quad\text{ in }\quad\Omega,
\end{equation*}
where $\Omega$ is   a flat 2-dimensional torus $\mathbb{T}^2$ or $\mathbb{R}^2$, $\varepsilon> 0$ is a coupling parameter, and $\delta_p$ denotes the Dirac measure concentrated
at $p$.
In this paper, we prove that, when the coupling parameter $\varepsilon$ is small, the topological type solutions to the above system are uniquely determined by the location of their vortex points. This result follows by the bubbling analysis and the non-degency of linearized equations.

\end{abstract}

\date{\today}
\maketitle
\begin{section}{Introduction}

In recent years,  various Chern--Simons models have been proposed to study condensed matter physics and particle physics, such as the relativistic Chern-Simons models of high temperature superconductivity \cite{JW,Dunne1}, Lozano-Marqu\'es-Moreno-Schaposnik model \cite{LMMS} of bosonic sector of  $\mathcal{N}=2$ supersymmetric  Chern-Simons-Higgs theory,  and Gudnason model \cite{Gu1, Gu2} of $\mathcal{N}=2$ supersymmetric  Yang-Mills-Chern-Simons-Higgs theory and so on. The relative Euler--Lagrange equations of those models often provided many mathematical challenging problems. We refer the readers to \cite{Dunne1,yangbook}
for exhaustive bibliography.\par

Speilman et al.\cite{SFEG} observed no parity breaking in the experiment with high temperature superconductivity. Hagen\cite{Hagen} and  Wilczek\cite{Wil}  indicated the parity broken  may not  happen in the a field theory with even number of Chern-Simons gauge fields.  One of the simplest models of this kind is the $[U(1)]^2$ Chern-Simons model of two Higgs fields, where each of them coupled to one of two Chern--Simons fields. In this paper, we will study  the relativistic self-dual  $[U(1)]^2$  Chern-Simons model proposed by Kim et al\cite{KLKLM}. For simplicity, we  consider the case with only mutual Chern--Simons interaction. We give only a brief description on this model. Let   $(A_\mu^{(i)})$ $(\mu=0,1, 2,\,i=1,2)$ be two Abelian gauge fields and $\phi_i$ $(i=1,2)$  be two Higgs scalar fields, where the electromagnetic fields and covariant derivatives are defined by
\begin{equation}
   F^{(i)}_{\mu\nu}=\partial_\mu A_\nu^{(i)}-\partial_\nu A_\mu^{(i)},\quad D_\mu\phi_i=\partial_\mu\phi_i-\mathrm{i} A^{(i)}_\mu\phi_i, \quad \mu=0, 1,2,\,\,i=1,2. \label{a2}
\end{equation}
The Lagrangian of this model  is  written in the form
\begin{equation}
 \mathcal{L}  =  -\frac{\varepsilon}{ 2}\epsilon^{\mu\nu\alpha}\left(A^{(1)}_\mu F^{(2)}_{\mu\nu}+A^{(2)}_\mu F^{(1)}_{\mu\nu}\right)+\sum\limits_{i=1}^2D_\mu\phi_i\overline{D^\mu\phi_i}-V(\phi_1, \phi_2), \label{a1}
\end{equation}
where $\varepsilon>0$ is a coupling parameter,
and the Higgs potential $ V(\phi_1,\phi_2) $ is  taken as
 \begin{equation}
  V(\phi_1,\phi_2) =    \frac{1}{4\varepsilon^2}\left(|\phi_2|^2\left[|\phi_1|^2-1\right]^2+|\phi_1|^2\left[|\phi_2|^2-1\right]^2\right).
\end{equation}
 After a BPS reduction \cite{Bo,PS}, one can show that the energy minimizer satisfies the following self-dual equation:
  \begin{equation}\label{er3}\left\{
\begin{split}
 &D_1\phi_k\pm\mathrm{i} D_2\phi_k=0, \quad k=1, 2 , \\
 & F^{(1)}_{12} \pm \frac{1}{2\varepsilon^2}|\phi_2|^2\left(|\phi_1|^2-1\right)=0,\\
  & F^{(2)}_{12} \pm \frac{1}{2\varepsilon^2}|\phi_1|^2\left(|\phi_2|^2-1\right)=0.
   \end{split}\right.
\end{equation}

As in \cite{JT}, we let  $u_{\varepsilon, i}=\ln |\phi_i|^2$, and denote the zeros of $\phi_i$ by $\{p_{1,i}, \dots, p_{{N_i},i}\}$,
 $i=1, 2$. Then $(u_{\varepsilon, 1},u_{\varepsilon,2})$ satisfies
\begin{equation}
\begin{aligned}\label{maineq}
\left \{
\begin{array}{ll}
\Delta u_{\varepsilon,1}+\frac{1}{\varepsilon^2} e^{u_{\varepsilon,2}}(1-e^{u_{\varepsilon,1}})=4\pi \sum^{N_1}_{j=1} \delta_{p_{j,1}}\quad\mbox{on }~ \Omega,
\\
\\
\Delta u_{\varepsilon,2}+\frac{1}{\varepsilon^2} e^{u_{\varepsilon,1}}(1-e^{u_{\varepsilon,2}})=4\pi \sum^{N_2}_{j=1} \delta_{p_{j,2}}\quad\mbox{on }~ \Omega,
\end{array}\right.
\end{aligned}
\end{equation}
where $\delta_p$ is  the Dirac measure   at $p$.  See \cite{KLKLM,dz94,LPY} for the details of the derivation of \eqref{maineq} from \eqref{er3}. $\Omega$ here is usually refereed to  $\mathbb{R}^2$ or a flat tours $\mathbb{T}^2$.  \par

When $u_{\varepsilon,1}=u_{\varepsilon,2}=u_{\varepsilon}$ and  $\sum^{N_1}_{j=1}\delta_{p_{j,1}}=\sum^{N_2}_{j=1}\delta_{p_{j,2}}=\sum^{N}_{j=1}\delta_{p_{j }}$, then
\eqref{maineq} is reduced to
\begin{equation}\label{reeq}
\Delta u_{\varepsilon}+\frac{1}{\varepsilon^2}e^{u_{\varepsilon}}(1-e^{u_{\varepsilon}})=4\pi \sum^{N }_{j=1} \delta_{p_{j }},
\end{equation}
which is the equation derived from the Abelian Chern-Simons model with one Higgs particles. See \cite{JW} for the physical background. Compared to \eqref{maineq}, the equation \eqref{reeq} has been studied extensively in the last two decades. We refer
\cite{CY1,CFL,C,choe,CKL,Dunne1,SY,T1,T} and reference therein for more details.\par

On the other hand, the system \eqref{maineq} is  a typical {\it skew-symmetric system}.  We introduce the background functions on $\mathbb{T}^2$ to remove the singularities.
 \begin{equation}
  u_{0,i}=-4\pi \sum_{j=1}^{N_i}G(x, p_{j,i}), \quad i=1, 2,
\end{equation}
 where   $G(x, q)$ is  the Green function defined by
\begin{equation}\label{greeneq}\left\{
\begin{split}
&-\Delta G(x, q)=\delta_q-\frac{1}{|\mathbb{T}^2|},\\
&\int_{\mathbb{T}^2} G(x,q)dx=0,
\end{split}\right.
\end{equation}
and  $|\mathbb{T}^2| $ is the area of $\mathbb{T}^2$.  With the transform $u_{\varepsilon, i}\to u_{0,i}+u_{\varepsilon, i}, i=1,2$, the system \eqref{maineq} can be reduced into
    \begin{equation}\label{nosingeq}\left\{
\begin{split}
\Delta u_{\varepsilon,1}+\frac{1}{\varepsilon^2}e^{u_{0,2}+u_{\varepsilon, 2}}(1-e^{u_{0,1}+u_{\varepsilon, 1}})= \frac{ 4N_1\pi}{|\mathbb{T}^2|}\\
\Delta u_{\varepsilon, 2}+\frac{1}{\varepsilon^2}e^{u_{0,1}+u_{\varepsilon, 1}}(1-e^{u_{0,2}+u_{\varepsilon, 2}})= \frac{4N_2\pi}{|\mathbb{T}^2|}
\end{split} \quad \quad\quad\text{   in  }\,\,\mathbb{T}^2.\right.
\end{equation}
Then any solution of the system \eqref{nosingeq} is   a critical point of the following functional
 \begin{equation}\label{functional}
\begin{aligned}
 &  I(u_{\varepsilon, 1}, u_{\varepsilon, 2})\\
 & =\int_{\mathbb{T}^2} \Big\{\frac12 \nabla u_{\varepsilon, 1}\cdot \nabla u_{\varepsilon, 2}+\frac{1}{\varepsilon^2} (1-e^{u_{0,1}+u_{\varepsilon, 1}})(1-e^{u_{0,2}+u_{\varepsilon, 2}})
 \\&\quad+\frac{4\pi}{|\mathbb{T}^2|}(N_1u_{\varepsilon, 1}+N_2u_{\varepsilon, 2})\Big\}  dx
  \end{aligned}
 \end{equation}
We refer the readers to \cite{Yan1,Yan2,Yan3} for more information about  skew-symmetric systems. Since  the  action functional  \eqref{functional} is indefinite, there are  difficulties of  studying \eqref{maineq} from the direct variational method. \par

From the potential energy density, it can be seen that the finite energy condition impose the following behaviors of
$(u_{\varepsilon, 1},u_{\varepsilon,2})$:

\begin{itemize}
\item[{\bf a.}] $\Omega=\mathbb{R}^2$

 \subitem (1)  $(u_{\varepsilon, 1},u_{\varepsilon,2})\to (0,0)$ \quad\quad  \;  as   $|x|\to \infty$.

 \subitem (2)  $(u_{\varepsilon, 1},u_{\varepsilon,2})\to (-\infty,-\infty)$ as $|x|\to \infty$.
\item[{\bf b.}] $\Omega=$ flat torus  $\mathbb{T}^2$
 \subitem (1)  $(u_{\varepsilon, 1},u_{\varepsilon,2})\to (0,0)$ \quad\quad  \;  \;  a.e.  as $\varepsilon\to 0$.
 \subitem (2)  $(u_{\varepsilon, 1},u_{\varepsilon,2})\to (-\infty,-\infty)$\quad  a.e. as $\varepsilon\to 0$.
\end{itemize}
In the physical literature, a solution to \eqref{maineq} satisfies ${\bf a.}(1)$ or ${\bf b.}(1)$ is  called a {\it topological solution} and satisfies ${\bf a.}(2)$ or ${\bf b.}(2)$ is called a {\it non-topological solution}. \par

Lin, Ponce and Yang \cite{LPY} initiated the mathematical study on this system, where they established the existence of topological solution in $\mathbb{R}^2$. Since the main difficulty arises from the skew-symmetric structure of the system \eqref{maineq}, they used the constrained minimization method with a deep application of Moser-Trudinger inequality. Since then, this system \eqref{maineq} has been studied from other aspects, such as the existence of topological  solutions over  a flat torus \cite{LP},  the existence of  non-toplogical solutions over the  plane  and a flat torus \cite{HHL1} and the structure of the radial solutions  over the plane \cite{CCL}.\par

In \cite{LP}, Lin and  Prajapat applied a monotone scheme and the constrained minimization method  to obtain two kind of solutions to \eqref{maineq} over a flat torus: maximal solution and mountain-pass solution. Here,   $(u_{\varepsilon,1}, u_{\varepsilon,2})$ is called a maximal solution to \eqref{maineq} if
$$u_i<u_{\varepsilon,i},\quad i=1,2,$$
for other solutions $(u_1,u_2)$ to \eqref{maineq}. Furthermore, they showed that the maximal solution is unique when $\varepsilon>0$ is small.   It is obvious that  the  maximal solution is a topological solution.  Naturally, we are  lead to the question {\it whether the topological solution is unique.} We give a positive answer to this question when $\varepsilon>0$ is small.

\begin{theorem}\label{uniqueness} Consider $\Omega=\mathbb{T}^2$. There exists $\varepsilon_0:=\varepsilon_0(p_{j,i})>0$ such that
 there exists a unique topological solution of \eqref{maineq} for each $\varepsilon\in(0,\varepsilon_0)$. Moreover, any topological solution is a unique maximal solution of \eqref{maineq} for  $\varepsilon\in(0,\varepsilon_0)$.
\end{theorem}

It is worth to note that when  $u_{\varepsilon,1}=u_{\varepsilon,2}=u_{\varepsilon}$ and  $\sum^{N_1}_{j=1} \delta_{p_{j,1}}=\sum^{N_2}_{j=1} \delta_{p_{j,2}}=\sum^{N}_{j=1} \delta_{p_{j }}$, our theorem  is reduced to the uniqueness  theorem for the topological solutions  to the scalar equation \eqref{reeq} on $\mathbb{T}^2$  proved by Choe \cite{C}  (also on $\mathbb{R}^2$) and Tarantello \cite{T} independently. Choe \cite{C} showed that the topological solution can be approximated by the sum of rescaled radial topological solution and  used the invertibility of the linearized operator   from $W^{2,2}$ to $L^2$ to prove the uniqueness of the topological solutions.  Tarantello \cite{T} showed that the topological solutions to \eqref{reeq} is a strict local minimum for the corresponding action functional and the uniqueness follows.  On the other hand, since our problem has indefinite functional, it is difficult to use the concept of stability (local minimizer). So  we use different approach such that we observe the behavior of the direct difference of two topological solutions with $L^\infty$-normalization instead of $W^{2,2}$ or $L^2$. In our proof of Theorem \ref{uniqueness}, to prove  the uniqueness of topological solution for small $\varepsilon$ on $\mathbb{T}^2$, we investigate  the behavior of topological solutions as $\varepsilon\to0$ for \eqref{maineq} in Section 2 as a generalization of the estimates obtained in \cite{C,T}.  In fact,   the similar arguments on $\mathbb{T}^2$ in Section 2 also hold for the topological  entire solutions on $\mathbb{R}^2$ due to the fact that the topological entire solutions achieve the boundary condition exponentially fast at infinity. More precisely,  Lin Ponce and Yang proved the following theorem.\\
\\
{\bf Theorem A.} \cite{LPY} {\it Suppose  $(u_{\varepsilon,1},u_{\varepsilon,2})$ is a topological    solution of \eqref{maineq}  in $\mathbb{R}^2$. Then
\begin{equation}
\sum_{i=1}^2  ( |u_{\varepsilon,i}(x)| + |\nabla u_{\varepsilon,i}(x)| )\leq C\frac{e^{-\frac{|x|}{\varepsilon}}}{|x|^{1/2}}\end{equation}
for some constant $C$ and  $|x|$ sufficiently large. }

\bigskip
In view of the above good exponential decay property of the topological entire solutions,   we obtain the following theorem.
\begin{theorem}\label{entire} Consider $\Omega=\mathbb{R}^2$. There exists $\varepsilon_0:=\varepsilon_0(p_{j,i})>0$ such that
 there exists a unique topological entire solution of \eqref{maineq} for each $\varepsilon\in(0,\varepsilon_0)$.
\end{theorem}\par

Firstly, we sketch our proof for Theorem \ref{uniqueness} here. Suppose, for the sake of contradiction, that  there exist two sequences of distinct topological  solutions $(u_{\varepsilon,1},u_{\varepsilon,2})$ and $(\tilde{u}_{\varepsilon,1},\tilde{u}_{\varepsilon,2})$ of  \eqref{maineq}.
 Without loss of generality, we may assume that there exists $x_\varepsilon\in\mathbb{T}^2$ such that
 \begin{equation*}
 |u_{\varepsilon,1}(x_\varepsilon)-\tilde{u}_{\varepsilon,1}(x_\varepsilon)|=\|u_{\varepsilon,1}-\tilde{u}_{\varepsilon,1}\|_{L^\infty(\mathbb{T}^2)}\ge\|u_{\varepsilon,2}-\tilde{u}_{\varepsilon,2}\|_{L^\infty(\mathbb{T}^2)}
 \end{equation*}
 and  $x_\varepsilon\to p$ as $\varepsilon\to 0$(up to subsequence).
 Set
\begin{equation} \label{AB}
  A_\varepsilon\equiv\frac{u_{\varepsilon,1}-\tilde{u}_{\varepsilon,1}}{\|u_{\varepsilon,1}-\tilde{u}_{\varepsilon,1}\|_{L^\infty(\mathbb{T}^2)}} \quad  \text{and} \quad B_\varepsilon\equiv\frac{u_{\varepsilon,2}-\tilde{u}_{\varepsilon,2}}{\|u_{\varepsilon,1}-\tilde{u}_{\varepsilon,1}\|_{L^\infty(\mathbb{T}^2)}} .
 \end{equation}
 Then $(A_\varepsilon,B_\varepsilon)$ satisfies
\begin{equation}\label{linear-sym-1}
\begin{aligned}
\left \{
\begin{array}{ll}
\Delta A_\varepsilon-\frac{1}{\varepsilon^2} e^{\tilde{u}_{\varepsilon,2}+\eta_{\varepsilon,1}}A_\varepsilon+\frac{1}{\varepsilon^2} e^{\eta_{\varepsilon,2}}(1-e^{u_{\varepsilon,1}})B_\varepsilon=0\quad\mbox{on }~ \mathbb{T}^2,
\\
\\
\Delta B_\varepsilon-\frac{1}{\varepsilon^2} e^{\tilde u_{\varepsilon,1}+\eta_{\varepsilon,2}}B_\varepsilon+\frac{1}{\varepsilon^2} e^{\eta_{\varepsilon,1}}(1-e^{u_{\varepsilon,2}})A_\varepsilon=0\quad\mbox{on }~ \mathbb{T}^2,
\end{array}\right.
\end{aligned}
\end{equation}
where  $\eta_{\varepsilon,i}$ is between $u_{\varepsilon,i}$ and $\tilde u_{\varepsilon,i}$, $i=1,2.$
 After suitable rescaling at the maximum points (see Section 3), \eqref{linear-sym-1} converges to a bounded solution $(A,B)$  of
 \begin{equation}\label{linear-sym-2}
\left\{
\begin{array}{l}
\Delta A - e^{U_1+U_2} A+ e^{U_2}(1-e^{U_1})B  =0 \\
\Delta B -  e^{U_1+U_2}B +  e^{U_1}(1-e^{U_2})A  =0
\end{array}\right.
\mbox{   in   }\mathbb{R}^2.
\end{equation}
where $(U_1,U_2)$ is a topological   solution to
\begin{equation}\label{radialeq}
\left\{
\begin{array}{l}
\Delta u_1 +e^{u_2}(1-e^{u_1}) =4\pi \nu_1 \delta_0 \\
\Delta u_2  +e^{u_1}(1-e^{u_2})  = 4\pi \nu_2\delta_0
\end{array}\right.
\mbox{   in   }\mathbb{R}^2.
\end{equation}
 and $\nu_1$ and $\nu_2$ are constants which are determined by  the choice of the rescaling region.
By the standard method of moving plane\cite{BS}, one can show that  the topological solution $(u_1,u_2)$ of \eqref{radialeq} is radially symmetric with respect to the origin.  For any topological solution of \eqref{radialeq},
 Chern, Chen and Lin\cite{CCL} showed the  non-degeneracy of the linearized system \eqref{linear-sym-2}, i.e., $A=B=0$.\\
\\
{\bf Theorem B. }\cite{CCL} {\it  Let $(U_1, U_2)$ be the radial topological  solution of   \eqref{radialeq}. Then the linearized
equation   \eqref{linear-sym-2} of  \eqref{radialeq} at $(U_1, U_2)$  is non-degenerate, i.e.,
if $(A,B)$ is a pair of bounded solution of  \eqref{linear-sym-2}, then $$(A,B)\equiv (0,0).$$ Moreover, equation     \eqref{radialeq}
possesses one and only one topological solution.}

\bigskip

Then the uniqueness of the topological solutions of \eqref{maineq} on  $\mathbb{T}^2$  follows  from Theorem B.  Obviously, the main ingredient of our approach is to show how
$(A_\varepsilon,B_\varepsilon)$ would converge to a bounded solution of \eqref{linear-sym-2}.

 The main different part between the proof of Theorem \ref{uniqueness} and
Theorem \ref{entire} is that on $\mathbb{R}^2$, the maximum point $x_\varepsilon$ of $|u_{\varepsilon,i}-\tilde{u}_{\varepsilon,i}|$ can diverge to $\infty$ unlike on $\mathbb{T}^2$. Even in this case, we can use the good convergence property of topological entire solutions to prove Theorem \ref{entire} (see the end of Section 3).


This paper is organized as follows. In Section 2, we establish some preliminary estimates for the topological solutions which are important to show that \eqref{linear-sym-1} converges to \eqref{linear-sym-2}.   Section 3 is devoted to the proof of Theorem \ref{uniqueness}-\ref{entire}.

\end{section}
\begin{section}{Preliminaries}
In this section,   we will show some preliminary estimates for the topological solutions to \eqref{maineq} in $\mathbb{T}^2$. Then in view of Theorem A, the similar arguments in this section are also true for the topological entire solutions on $\mathbb{R}^2$.
Our main goal in this section is to show the topological solutions to \eqref{maineq}, after suitable rescaling, can converge to the radially symmetric entire topological solutions  on a certain domain(see Lemma \ref{linfty} below).

By maximum principle, it is clear that satisfies $u_{\varepsilon,i}<0$  on $\mathbb{T}^2$  for $i=1,2$.
Thus, by integrating \eqref{maineq}, we have
\begin{equation}
\begin{aligned}\label{bddofintegration}
\int_{\mathbb{T}^2} \frac{1}{\varepsilon^2} e^{u_{\varepsilon,j}}(1-e^{u_{\varepsilon,i}})dx=\int_{\mathbb{T}^2} \frac{1}{\varepsilon^2} |e^{u_{\varepsilon,j}}(1-e^{u_{\varepsilon,i}})|dx=4\pi N_i,\quad  1\le j\neq i\le2.
\end{aligned}
\end{equation} \par

We show that, there are only two types of solutions, topological and non-topological solutions, as $\varepsilon\to 0$. In particular, if $(u_{\varepsilon,1},u_{\varepsilon,2})$ is a topological solution, then $u_{\varepsilon,i}\to 0$($i=1,2$) in $L^{p}(\mathbb{T}^2)$ for some $p>1$ (In fact, we can improve the convergence result for topological solutions in Lemma \ref{speedofconvergence}).

\begin{lemma}\label{Lp}
Let $(u_{\varepsilon,1},u_{\varepsilon,2})$ be a sequence of solutions of (\ref{maineq}).  Then, up to  subsequence,
one of the following holds true:

(i)  for $i=1,2$, $u_{\varepsilon,i}\to-\infty$ a.e. as $\varepsilon\to0$;

(ii) for $i=1,2$, $u_{\varepsilon,i}\to0$ a.e. as $\varepsilon\to0$. Moreover, $u_{\varepsilon,i}\to0$  in $L^p(\mathbb{T}^2)$ for some $p>1$, $i=1,2$.
\end{lemma}
\begin{proof}
By \eqref{bddofintegration}, $e^{u_{\varepsilon,j}}(1-e^{u_{\varepsilon,i}})\to0$ in $L^1(\mathbb{T}^2)$ as $\varepsilon\to 0$. Hence,  it is clear that  either $u_{\varepsilon,i}\to-\infty$ a.e. or $u_{\varepsilon,i}\to 0 $ a.e. for $i=1,2$. So, we only need to show the $L^p$ estimate in (ii).\par

Let $d_{\varepsilon,i}=\frac{1}{|\mathbb{T}^2|}\int_{\mathbb{T}^2} u_{\varepsilon,i} dx$ and $u_{\varepsilon,i}=w_{\varepsilon,i}+u_{0,i} +d_{\varepsilon,i}$. Then
 $(w_{\varepsilon,1},w_{\varepsilon,2})$ satisfies
\begin{equation}
\begin{aligned}\label{wemain}
\left \{
\begin{array}{ll}
\Delta w_{\varepsilon,1}+\frac{1}{\varepsilon^2} e^{u_{\varepsilon,2}}(1-e^{u_{\varepsilon,1}})=\frac{4\pi N_1}{|\mathbb{T}^2|}\quad\mbox{on }~ \mathbb{T}^2,
\\
\\
\Delta w_{\varepsilon,2}+\frac{1}{\varepsilon^2} e^{u_{\varepsilon,1}}(1-e^{u_{\varepsilon,2}})=\frac{4\pi N_2}{|\mathbb{T}^2|}\quad\mbox{on }~ \mathbb{T}^2,
\end{array}\right.
\end{aligned}
\end{equation}
and  $\int_{\mathbb{T}^2} w_{\varepsilon,i} dx=0$, $i=1,2.$\par
We claim that  there exist $C_q>0$ such that $\|\nabla w_{\varepsilon,i}\|_{L^q(\mathbb{T}^2)}\le C_q$ for any $q\in(1,2)$.
Let $q'=\frac{q}{q-1}>2$. Then
\begin{equation}\begin{aligned}\label{normexpression}
&\|\nabla w_{\varepsilon,i}\|_{L^q(\mathbb{T}^2)}
\\&\le\sup\Big\{\Big|\int_{\mathbb{T}^2}\nabla w_{\varepsilon,i}\nabla\phi dx\Big|\  \Big|\ \ \phi\in W^{1,q'}(\mathbb{T}^2),\ \int_{\mathbb{T}^2}\phi dx=0,\ \|\phi\|_{W^{1,q'}(\mathbb{T}^2)}=1\Big\}.
\end{aligned}\end{equation}
By lemma 7.16 in \cite{GT}, if  $\int_{\mathbb{T}^2}\phi dx=0$, then there exist $c,\ C>0$ such that
\begin{equation}\label{GTineq}
|\phi(x)|\le c\int_{\mathbb{T}^2}\frac{|\nabla\phi|}{|x-y|}dy\le C\|\nabla\phi\|_{L^{q'}(\mathbb{T}^2)}\ \ \textrm{for}\ x\in\mathbb{T}^2.
\end{equation}
Thus  in view of  (\ref{wemain}),  (\ref{GTineq}), and  \eqref{bddofintegration}, we see that there exists constant $C>0$, independent of $\phi$ satisfying $\int_{\mathbb{T}^2}\phi dx=0,\ \|\phi\|_{W^{1,q'}(\mathbb{T}^2)}=1$,
\begin{equation}
\begin{aligned}
\Big|\int_{\mathbb{T}^2}\nabla w_{\varepsilon,i}\nabla\phi dx\Big|&=\Big|\int_{\mathbb{T}^2}\Delta w_{\varepsilon,i}\phi dx\Big|
\\&\le\|\phi\|_{L^\infty(\mathbb{T}^2)}\Big|\int_{\mathbb{T}^2} |\frac{1}{\varepsilon^2} e^{u_{\varepsilon,j}}(e^{u_{\varepsilon,i}}-1)|dx+4\pi N_i\Big|\le C.
\end{aligned}
\end{equation}
Now using (\ref{normexpression}), we complete the proof of our claim.\par
In view of  Poincar\'{e} inequality, we also have $\|w_{\varepsilon,i}\|_{L^q(\mathbb{T}^2)}\le c\|\nabla w_{\varepsilon,i}\|_{L^q(\mathbb{T}^2)}$. Then there exist  $w_i\in W^{1,q}(\mathbb{T}^2)$ and $p>1$ such that, as $\varepsilon\to0$,
\begin{equation}\label{wecone}
w_{\varepsilon,i}\rightharpoonup w_i\ \ \textrm{weakly in}\  W^{1,q}(\mathbb{T}^2),\ w_{\varepsilon,i}\to w_i\ \ \textrm{strongly in}\  L^p(\mathbb{T}^2),\  w_{\varepsilon,i}\to w_i\  \textrm{a.e.}.
\end{equation}
We consider the following possible cases.
\begin{itemize}
\item[(i)] $\limsup_{\varepsilon\to0}\frac{e^{d_{\varepsilon,i}}}{\varepsilon^2}\leq c$ for some constant $c>0$.
\item[(ii)] $\limsup_{\varepsilon\to0}\frac{e^{d_{\varepsilon,i}}}{\varepsilon^2}=+\infty$.
\end{itemize}

If $\limsup_{\varepsilon\to0}\frac{e^{d_{\varepsilon,i}}}{\varepsilon^2} $ is bounded, then
\begin{equation*}
e^{u_{\varepsilon,i}}=e^{w_{\varepsilon,i}+d_{\varepsilon,i}+u_{0,i}}\le c\varepsilon^2e^{w_{\varepsilon,i}+u_{0,i}}\to0\ \ \mbox{ a.e. as }\ \ \varepsilon\to0,
\end{equation*} which implies that
$u_{\varepsilon,i}\to-\infty$ a.e. as $\varepsilon\to0$.

Next, we consider the case
\begin{equation}\label{ass}
\limsup_{\varepsilon\to0}\frac{e^{d_{\varepsilon,i}}}{\varepsilon^2}=+\infty.
\end{equation}
Since $u_{\varepsilon,i}<0$ on $\mathbb{T}^2$, we see that $0\le e^{d_{\varepsilon,i}}\le1$ which implies there exists $A_i\ge0$ such that  $\limsup_{\varepsilon\to0}e^{d_{\varepsilon,i}}=A_i$.
By using Fatou's lemma and  (\ref{wecone}),  we see that
\begin{equation*}
\begin{aligned}
4\pi N_i\varepsilon^2&=
\int_{\mathbb{T}^2}e^{u_{\varepsilon,j}}(1-e^{u_{\varepsilon,i}})dx
\ge\int_{\mathbb{T}^2} A_je^{w_j+u_{0,j}}(1- A_ie^{w_i+u_{0,i}})dx,
\end{aligned}
\end{equation*}
which implies that $A_j\equiv0$ or $w_i+u_{0,i}=-\ln A_i$ a.e. in $\mathbb{T}^2$.\par

Let $1\leq i\neq j\leq 2$. If $A_j\equiv0$, then   $A_i\equiv0$ by the same argument.  Thus,  $\lim_{\varepsilon\to0}d_{\varepsilon,i}=-\infty$.  Moreover, by using \eqref{wecone}, we get that $u_{\varepsilon,i}=w_{\varepsilon,i}+d_{\varepsilon,i}+u_{0,i}\to -\infty$ a.e. in $\mathbb{T}^2$ for $i=1,2$.

If $w_i+u_{0,i}=-\ln A_i$ a.e. in $\mathbb{T}^2$, then in view of  $\int_{\mathbb{T}^2} w_{\varepsilon,i}+u_{0,i}dx=0$, we see that $A_i=1$ and $w_i+u_{0,i}=0$ a.e. in $\mathbb{T}^2$.
Thus, $u_{\varepsilon,i}=w_{\varepsilon,i}+d_{\varepsilon,i}+u_{0,i}\to w_i+\ln A_i+u_{0,i}=0$ a.e. in $\mathbb{T}^2$ and $u_{\varepsilon,i}\to0$ in $L^p(\mathbb{T}^2)$ for some $p>1$. Now we complete the proof of Lemma \ref{Lp}.
\end{proof}

In the following lemma, we get the detailed information about topological solutions.
\begin{lemma}\label{speedofconvergence}
Let $(u_{\varepsilon,1}, u_{\varepsilon,2})$ be a sequence of topological solutions of (\ref{maineq})
Then we have as $\varepsilon\to0$,

$(i)$  $u_{\varepsilon,i}\to0$ in $C^m_{\textrm{loc}}(\mathbb{T}^2\setminus Z)$ for any $m\in \mathbb{Z}^+$ and faster than any other power of $\varepsilon>0$;

$(ii)$ for $1\le j\neq i\le2$, $\frac{1}{\varepsilon^2} e^{u_{\varepsilon,j}}(1-e^{u_{\varepsilon,i}})\to4\pi\sum^{N_{i}}_{j=1}\delta_{p_{j,i}}$  and

 $\frac{1}{\varepsilon^2} (1-e^{u_{\varepsilon,1}})(1-e^{u_{\varepsilon,2}})\to4\pi\Big(\sum^{N_{1}}_{j=1}\delta_{p_{j,1}}\Big)\Big(\sum^{N_{2}}_{j=1} \delta_{p_{j,2}}\Big)$ weakly in the sense of measure in $\mathbb{T}^2$.
\end{lemma}
\begin{proof}
Denote $\mathbb{T}^2_\delta\equiv\{x\in\mathbb{T}^2\ |\ \textrm{dist}(x,p_{j,i})\ge\delta\ \ \mbox{for all}\ \ i,j\}$.
Since $u_{\varepsilon,i}<0$ on $\mathbb{T}^2$, we note that $u_{\varepsilon,i}$ is subharmonic in $\mathbb{T}^2_\delta$. By using the mean value theorem and Lemma \ref{Lp},
we have that
\begin{equation}\label{linftybdd}
0\le-u_{\varepsilon,i}\le\frac{1}{|\mathbb{T}^2_\delta|}\|u_{\varepsilon,i}\|_{L^1(\mathbb{T}^2_\delta)}\to0\ \mbox{ as} \ \varepsilon\to0\ \mbox{ on }\  \mathbb{T}^2_{2\delta}.
\end{equation}
We also have the following inequality,
\begin{equation}\label{ele}
\frac{|t|}{1+|t|}\le|1-e^t|\ \ \textrm{for any }\ t\in\mathbb{R}.
\end{equation}
By using   \eqref{bddofintegration} and \eqref{linftybdd}, we deduce the estimate
\begin{equation}
\begin{aligned}\label{l1}
\int_{\mathbb{T}^2_\delta}|u_{\varepsilon,i}|dx&\le(1+\|u_{\varepsilon,i}\|_{L^\infty(\mathbb{T}^2_\delta)})\int_{\mathbb{T}^2_\delta}\frac{|u_{\varepsilon,i}|}{1+|u_{\varepsilon,i}|}dx
\\&\le(1+\|u_{\varepsilon,i}\|_{L^\infty(\mathbb{T}^2_\delta)})e^{\|u_{\varepsilon,j}\|_{L^\infty(\mathbb{T}^2_\delta)}}\int_{\mathbb{T}^2_\delta}e^{u_{\varepsilon,j}}(1-e^{u_{\varepsilon,i}})dx
\\&\le8\pi e N_i\varepsilon^2.
\end{aligned}
\end{equation}
Let $\phi\in C^\infty(\overline{\mathbb{T}^2})$ satisfy
\begin{equation}\left\{
\begin{split}
\phi=0 &\qquad\qquad\text{ in } \mathbb{T}^2\setminus \mathbb{T}^2_{\delta} \\
\phi=1 &\qquad\qquad\text{ in } \mathbb{T}^2_{2\delta}\\
\end{split}\right.
\end{equation}
and $0\leq \phi\leq 1$. By using  $\phi$ as a test function in (\ref{maineq}) and (\ref{l1}), we  get that
\begin{equation}
\begin{aligned}\label{usingl1}
\frac{1}{\varepsilon^2}\int_{\mathbb{T}^2_{2\delta}}e^{u_{\varepsilon,j}}(1-e^{u_{\varepsilon,i}})dx&\le
\frac{1}{\varepsilon^2}\int_{\mathbb{T}^2}e^{u_{\varepsilon,j}}(1-e^{u_{\varepsilon,i}})\phi dx
\\&=-\int_{\mathbb{T}^2}\Delta u_{\varepsilon,i}\phi dx=
-\int_{\mathbb{T}^2} u_{\varepsilon,i}\Delta\phi dx\\&\le c_\delta\|u_{\varepsilon,i}\|_{L^1(\mathbb{T}^2_\delta)}\le C_\delta\varepsilon^2,
\end{aligned}
\end{equation}
for some constants $c_\delta,\ C_\delta>0$.
By a suitable iteration of (\ref{l1}), (\ref{usingl1}), and the elliptic estimates, we deduce that $(i)$ holds.
In other words,
for any small $\delta>0$  and any $m, n \in \mathbb{Z}^+$,  there exists a constant $c_{\delta,m,n}>0$ such that
\begin{equation}\label{conto0}\sup_{\mathbb{T}^2_{2\delta}}\Big(\sum_{|\alpha|=0}^m|D^\alpha u_{\varepsilon,i}|\Big)\le c_{\delta,m,n}\varepsilon^n.
\end{equation}

Next, if we take $\phi\in C^\infty(\mathbb{T}^2)$ as a test function into (\ref{maineq}), from Lemma \ref{Lp}, we see that
\begin{equation}\begin{aligned}
&\Big|\int_{\mathbb{T}^2}  \frac{1}{\varepsilon^2} e^{u_{\varepsilon,j}}(1-e^{u_{\varepsilon,i}})\phi dx-4\pi\sum^{N_{i}}_{j=1} \phi(p_{j,i})\Big|
\\&=\Big|\int_{\mathbb{T}^2}-\Delta u_{\varepsilon,i}\phi dx\Big|=\Big|\int_{\mathbb{T}^2}-u_{\varepsilon,i}\Delta\phi dx\Big|
\\&\le\|\phi\|_{C^2(\mathbb{T}^2)}\|u_{\varepsilon,i}\|_{L^1(\mathbb{T}^2)}\to0\ \ \textrm{as}\ \ \varepsilon\to0.
\end{aligned}
\end{equation}

Choose small $r>0$ such that $B_r(p_{j,i})\cap B_r(p_{j',i'})=\emptyset$ if $p_{j,i}\neq p_{j',i'}$ and  let
$$v_{\varepsilon,i}(x)=u_{\varepsilon,i}(x)-2\nu_{i}\ln|x-p|\qquad\text{ on }\qquad B_r(p),$$ where  $\nu_i=0$ if  $p\notin\cup_{j=1}^{d_i}\{p_{j,i}\}$ and $\nu_i=\#\{p_{j,i}|p_{j,i}=p \}$. Then $v_{\varepsilon,i}$ satisfies
\begin{equation}\label{maineqq}
\Delta v_{\varepsilon,i}+\frac{1}{\varepsilon^2} e^{u_{\varepsilon,j}}(1-e^{u_{\varepsilon,i}})=0\ \ \textrm{on}\ \ B_r(p).
\end{equation}
For the sake of simplicity, we assume that $p=0$.
Multiplying (\ref{maineqq}) by $\nabla u_{\varepsilon,j}\cdot x$ $(1\le j\neq i\le 2)$ and integrating over $B_r(0)$ (see \cite{CCL}), we obtain the Pohozaev
type identity
\begin{equation*}
\begin{aligned}
&\int_{\partial B_r(0)}\Big[\frac{2(\nabla u_{\varepsilon,1}\cdot x)(\nabla u_{\varepsilon,2}\cdot x)}{|x|}-(\nabla u_{\varepsilon,1}\cdot\nabla u_{\varepsilon,2})|x|-\frac{1}{\varepsilon^2}(1- e^{u_{\varepsilon,1}})(1-e^{u_{\varepsilon,2}})|x|]d\sigma
\\&=-\int_{B_r(0)}\frac{2}{\varepsilon^2}(1- e^{u_{\varepsilon,1}})(1-e^{u_{\varepsilon,2}})dx+8\pi\nu_1\nu_2.
\end{aligned}
\end{equation*}
By using (\ref{conto0}), we have
$$\lim_{\varepsilon\to0}\int_{B_r(p)}\frac{2}{\varepsilon^2}(1- e^{u_{\varepsilon,1}})(1-e^{u_{\varepsilon,2}})dx=8\pi\nu_1\nu_2.$$
  Now we complete the proof of Lemma \ref{speedofconvergence}.
 \end{proof}

The existence of topological  solution (in fact, maximal solution) of \eqref{maineq} can be proved by Lemma \ref{Lp} and \cite[Theorem 1.1-(i),(ii)]{LP}.
Hence, to prove Theorem \ref{uniqueness}, it suffices to prove the uniqueness property.
 To prove Theorem \ref{uniqueness}, we argue by contradiction and suppose that there exist two sequences of distinct topological solutions $(u_{\varepsilon,1},u_{\varepsilon,2})$ and $(\tilde{u}_{\varepsilon,1},\tilde{u}_{\varepsilon,2})$ of (\ref{maineq}).
 Without loss of generality, we may assume that there exists $x_\varepsilon\in\mathbb{T}^2$  such that
 \begin{equation*}
 |u_{\varepsilon,1}(x_\varepsilon)-\tilde{u}_{\varepsilon,1}(x_\varepsilon)|=\|u_{\varepsilon,1}-\tilde{u}_{\varepsilon,1}\|_{L^\infty(\mathbb{T}^2)}\ge\|u_{\varepsilon,2}-\tilde{u}_{\varepsilon,2}\|_{L^\infty(\mathbb{T}^2)},
 \end{equation*}
 and $x_\varepsilon\to p $ for some $p$ in $\mathbb{T}^2$.
 Set $A_\varepsilon\equiv\frac{u_{\varepsilon,1}-\tilde{u}_{\varepsilon,1}}{\|u_{\varepsilon,1}-\tilde{u}_{\varepsilon,1}\|_{L^\infty(\mathbb{T}^2)}}$ and $B_\varepsilon\equiv\frac{u_{\varepsilon,2}-\tilde{u}_{\varepsilon,2}}{\|u_{\varepsilon,1}-\tilde{u}_{\varepsilon,1}\|_{L^\infty(\mathbb{T}^2)}} $.
 Then $(A_\varepsilon,B_\varepsilon)$ satisfies
\begin{equation}
\begin{aligned}\label{abmaineq}
\left \{
\begin{array}{ll}
\Delta A_\varepsilon-\frac{1}{\varepsilon^2} e^{\tilde{u}_{\varepsilon,2}+\eta_{\varepsilon,1}}A_\varepsilon+\frac{1}{\varepsilon^2} e^{\eta_{\varepsilon,2}}(1-e^{u_{\varepsilon,1}})B_\varepsilon=0\quad\mbox{on }~ \mathbb{T}^2,
\\
\\
\Delta B_\varepsilon-\frac{1}{\varepsilon^2} e^{\tilde u_{\varepsilon,1}+\eta_{\varepsilon,2}}B_\varepsilon+\frac{1}{\varepsilon^2} e^{\eta_{\varepsilon,1}}(1-e^{u_{\varepsilon,2}})A_\varepsilon=0\quad\mbox{on }~ \mathbb{T}^2,
\end{array}\right.
\end{aligned}
\end{equation}
where $\eta_{\varepsilon,i}$ is   between $u_{\varepsilon,i}$ and $\tilde{u}_{\varepsilon,i}$, $i=1,2$.
Choose small $r_0>0$ such that $B_{r_0}(p_{j,i})\cap B_{r_0}(p_{j',i'})=\emptyset$ if $p_{j,i}\neq p_{j',i'}$
We consider the scaled functions $$\hat{u}_{\varepsilon,i}(y)=u_{\varepsilon,i}(\varepsilon y+p),\ \ \bar{u}_{\varepsilon,i}(y)=\tilde{u}_{\varepsilon,i}(\varepsilon y+p)\ \ \textrm{in}\ B_{\frac{r_0}{\varepsilon}}(0)\equiv\Big\{y\in\mathbb{R}^2\ \Big|\ |y|<\frac{r_0}{\varepsilon}\Big\}.$$
Then both $(\hat{u}_{\varepsilon,1}, \hat{u}_{\varepsilon,2})$  and $(\bar{u}_{\varepsilon,1}, \bar{u}_{\varepsilon,2})$ are solutions of
\begin{equation*}
\begin{aligned}
\left \{
\begin{array}{ll}
\Delta u_{\varepsilon,1}+ e^{u_{\varepsilon,2}}(1-e^{u_{\varepsilon,1}})=4\pi \nu_1\delta_{0}\quad\mbox{on }~  B_{\frac{r_0}{\varepsilon}}(0),
\\
\\
\Delta u_{\varepsilon,2}+ e^{u_{\varepsilon,1}}(1-e^{u_{\varepsilon,2}})=4\pi \nu_2\delta_{0}\quad\mbox{on }~  B_{\frac{r_0}{\varepsilon}}(0),
\end{array}\right.
\end{aligned}
\end{equation*}
where  $\nu_i=0$ if  $p\notin\cup_{j=1}^{d_i}\{p_{j,i}\}$ and $\nu_i=\#\{p_{j,i}|p_{j,i}=p \}$.\par
We show  the  gradient estimate for the topological solutions to \eqref{maineq} in the following lemma.
\begin{lemma}\label{gradgrad} There exists a constant $c>0$, independent of $r>0$ and $\varepsilon>0$, such that  \begin{equation*}
\Big|\nabla \hat{u}_{\varepsilon,i}(x)-\frac{2\nu_ix}{|x|^2}\Big|+\Big|\nabla \bar{u}_{\varepsilon,i}(x)-\frac{2\nu_ix}{|x|^2}\Big|\le c \ \ \textrm{on}\ \ B_{\frac{r}{\varepsilon}}(0)\ \textrm{for}\ \ i=1,2.
\end{equation*}
\end{lemma}
\begin{proof}
We remind   the Green's function $G$ on $\mathbb{T}$   which satisfies
\begin{equation}
-\Delta_x G(x,y)=\delta_y-\frac{1}{|\mathbb{T}|},\ x, y\in \mathbb{T}\ \textrm{and}\ \int_\mathbb{T} G(x,y)dx=0.
\end{equation}
And we denote by  $\gamma(x,y)=G(x,y)+\frac{1}{2\pi}\ln|x-y|$ the regular part of $G$.
We also  recall that
 \begin{equation}
  u_{0,i}=-4\pi \sum_{j=1}^{N_i}G(x, p_{j,i}), \quad i=1, 2.
\end{equation}
Then by using the Green's representation formula for a solution $(u_{\varepsilon,1}, u_{\varepsilon,2})$ of (\ref{maineq}), we see that for $x\in \mathbb{T}$,
\begin{equation}
\begin{aligned}
u_{\varepsilon,i}(x)-u_{0,i}(x)
=\frac{1}{|\mathbb{T}|}\int_\mathbb{T} u_{\varepsilon,i}(y)dy+\int_\mathbb{T} G(x,y) \frac{1}{\varepsilon^2}e^{u_{\varepsilon,j}}(1-e^{u_{\varepsilon,i}})dy.
\end{aligned}
\end{equation}
Then we  see that for
 $x\in B_r(p)$,
 \begin{equation*}
\begin{aligned}
&\Big|\nabla u_{\varepsilon,i}(x)-\frac{2\nu_i (x-p)}{|x-p|^2}\Big|
\\&\le C+\frac{1}{2\pi\varepsilon^2}\int_{\mathbb{T}^2} \frac{e^{\hat{u}_{\varepsilon,j}}(1-e^{\hat{u}_{\varepsilon,i}})}{|x-y|}dy
\\&=C+\frac{1}{2\pi\varepsilon^2}\Big(\int_{B_\varepsilon(x)} \frac{e^{\hat{u}_{\varepsilon,j}}(1-e^{\hat{u}_{\varepsilon,i}})}{|x-y|}dy+\int_{\mathbb{T}^2\setminus B_\varepsilon(x)} \frac{e^{\hat{u}_{\varepsilon,j}}(1-e^{\hat{u}_{\varepsilon,i}})}{|x-y|}dy\Big)
\\&\le C+\frac{C'}{\varepsilon},
\end{aligned}
\end{equation*}
for some constants $C,\ C'>0$, independent of $r>0$ and $\varepsilon>0$.  The desired conclusion
follows by the substitution $x=\varepsilon x+p$, $\hat{u}_{\varepsilon,i}(x)=u_{\varepsilon,i}(\varepsilon x+p)$ and  $\bar{u}_{\varepsilon,i}(x)=\tilde{u}_{\varepsilon,i}(\varepsilon x+p)$.
\end{proof}
\begin{lemma}\label{linfty}
$\lim_{\varepsilon\to0}\sum_{i=1}^2\Big(\sup_{B_{\frac{r_0}{\varepsilon}}(0)}(|\hat{u}_{\varepsilon,i}-u_i|+|\bar{u}_{\varepsilon,i}-u_i|)\Big)=0$, where $(u_1,u_2)$ is  a unique topological solution of
\begin{equation}\begin{aligned} \label{u}
\left \{
\begin{array}{ll}
\Delta u_i+ e^{u_j}(1-e^{u_i})=4\pi\nu_i\delta_0 \ \textrm{in}\ \mathbb{R}^2,\  1\le j\neq i\le2; \\
\\ u_i<0, \ \sup_{\mathbb{R}^2\setminus B_1(0)}|\nabla u_i|<+\infty\ \ i=1,2;\\
\\ e^{u_2}(1-e^{u_1}),\  e^{u_1}(1-e^{u_2}), \  (1-e^{u_1})(1-e^{u_2})\in L^1(\mathbb{R}^2).
\end{array}\right.
\end{aligned}\end{equation}
\end{lemma}
\begin{proof}
We decompose \begin{equation}\label{decomposehatu}\hat{u}_{\varepsilon,i}(y)=2\nu_i\ln|y|+\hat{v}_{\varepsilon,i}(y).\end{equation}
Then $\hat{v}_{\varepsilon,i}$ $(i=1,2)$ satisfies
\begin{equation}\label{hatve}
\Delta \hat{v}_{\varepsilon,i}+ |y|^{2\nu_j} e^{\hat{v}_{\varepsilon,j}}(1-|y|^{2\nu_i}e^{\hat{v}_{\varepsilon,i}})=0
\ \textrm{in}\ B_{\frac{r_0}{\varepsilon}}(0),
\end{equation}
where $1\le j\neq i\le2$.
Since $\hat{u}_{\varepsilon,i}=2\nu_i\ln|y|+\hat{v}_{\varepsilon,i}<0$ on $B_{\frac{r_0}{\varepsilon}}(0)$, we  have
$$\hat{v}_{\varepsilon,i}\Big|_{\partial B_{R}(0)}< -2\nu_i\ln R\  \ \textrm{for any}\ R>0.$$
 By using the Green's representation formula for a solution $u_{\varepsilon,i}$ of (\ref{maineq})
(see Lemma \ref{gradgrad}), we see that there exists $c_0>0$ such that
\begin{equation}\label{gradgrad2}
|\nabla \hat{v}_{\varepsilon,i}(x)|\le c_0 \ \ \textrm{on}\ \ B_{\frac{r_0}{\varepsilon}}(0).
\end{equation}

We claim that $\hat{v}_{\varepsilon,i}$ is uniformly bounded in the $C^{2,\alpha}$ topology. To prove our claim,
we argue by contradiction and suppose that there exists $R_0>0$ such that
$$\lim_{\varepsilon\to0}\Big(\inf_{B_{R_0}(0)}\hat{v}_{\varepsilon,i}\Big)=-\infty.$$ Then (\ref{gradgrad2})
implies that $\lim_{\varepsilon\to0}\Big(\sup_{B_{R}(0)}\hat{v}_{\varepsilon,i}\Big)=-\infty$ for any  $R\ge R_0$.
Clearly Lemma \ref{speedofconvergence} shows that, for any $R\ge R_0$,
\begin{equation}
\begin{aligned}
\label{large}
8\pi (\nu_1\nu_2+\nu_i)&=\lim_{\varepsilon\to0}\int_{B_{\frac{r_0}{\varepsilon}}(0)}\frac{2}{\varepsilon^2}(1-e^{u_{\varepsilon,i}})dx\ge\lim_{\varepsilon\to0}\int_{B_R(0)}2(1-e^{\hat{u}_{\varepsilon,i}})dx
\\&=\lim_{\varepsilon\to0}\int_{B_R(0)}2(1-|x|^{2\nu_i}e^{\hat{v}_{\varepsilon,i}})dx
\ge\pi R^2.
\end{aligned}\end{equation}
Since the right hand side of (\ref{large}) could be arbitrarily large, we obtain a contradiction which  proves our claim.

Then we obtain a subsequence $\hat{v}_{\varepsilon,i}$ (still denoted in the same way) such that
\begin{equation}\label{convergencevep}
\hat{v}_{\varepsilon,i}\to v_i\ \textrm{uniformly in}\ C^2_{\textrm{loc}}(\mathbb{R}^2).
\end{equation}
Let us define $u_i(y)\equiv2\nu_i\ln|y|+v_i(y)$. In view of  (\ref{gradgrad2}),
\eqref{bddofintegration} and  Lemma \ref{speedofconvergence}, we see that $(u_1,u_2)$ satisfies (\ref{u}).
Since $\sup_{\mathbb{R}^2\setminus B_1(0)}|\nabla u_i|<+\infty$ and
$1-e^{u_i}\in L^1(\mathbb{R}^2)$, we  see that $(u_1,u_2)$ is a topological solution in $\mathbb{R}^2$.
Indeed, if there exists  a sequence $x_n\in\mathbb{R}^2$ such that,
$$\lim_{n\to\infty}|x_n|\to+\infty, \ \ \lim_{n\to\infty}u_i(x_n)=c\neq0,$$
then since $\sup_{|x|\ge1}|\nabla u_i(x)|\le C$,
there exist small $r_1>0$ and $c_0>0$, independent of $n$, such that
$$1-e^{u_i}\ge c_0>0\ \ \textrm{on}\ \ B_{r_1}(x_n).$$
Then $\int_{\mathbb{R}^2} (1-e^{u_i})dx\ge\sum_{n=1}^\infty\int_{B_{r_1}(x_n)}(1-e^{u_i})dx=+\infty$ which is a contradiction. Thus,  $(u_1,u_2)$ is a topological solution  of \eqref{u} in $\mathbb{R}^2$.

Moreover, by using a Pohozaev type identity (see Lemma \ref{speedofconvergence}),
we  have
\begin{equation}\label{intr}\int_{\mathbb{R}^2} (1-e^{u_1})(1-e^{u_2})dx=4\pi \nu_1\nu_2\ \ \mbox{and}\ \ \int_{\mathbb{R}^2} (1-e^{u_i})dx=4\pi (\nu_1\nu_2+\nu_i).\end{equation}
By \cite{LPY}, we also see that $u_i$ admits exponential decay at infinity.
Then in view of Lemma \ref{speedofconvergence}, \eqref{intr}, and the dominated convergence theorem, we get that
\begin{equation*}
\begin{aligned}
0&=\lim_{\varepsilon\to0}\int_{B_{\frac{r_0}{\varepsilon}}(0)}|e^{\hat{u}_{\varepsilon,i}}-e^{u_i}|dx
&\ge\frac{1}{2}\lim_{\varepsilon\to0}\int_{B_{\frac{r_0}{\varepsilon}}(0)}|\hat{u}_{\varepsilon,i}-u_i|dx,
\end{aligned}
\end{equation*}
which implies that \begin{equation*}
\lim_{\varepsilon\to0}\Big(\sup_{B_{\frac{r_0}{\varepsilon}}(0)}|\hat{u}_{\varepsilon,i}-u_i|\Big)=0
\end{equation*}
from  (\ref{convergencevep}),  $\sup_{\mathbb{R}^2\setminus B_1(0)}|\nabla u_i|<+\infty$, and  \eqref{gradgrad2}.

By Theorem B, we know that  a topological solution of \eqref{u} is unique. So, by applying the above arguments to $(\bar{u}_{\varepsilon,1},\bar{u}_{\varepsilon,2})$, we  complete the proof of Lemma \ref{linfty}.
\end{proof}

\section{ Proof of Theorem 1.1-1.2}   Firstly,   we focus on the proof of Theorem \ref{uniqueness}.

\textbf{Proof of Theorem \ref{uniqueness}}
Recall that
$$
  A_\varepsilon\equiv\frac{u_{\varepsilon,1}-\tilde{u}_{\varepsilon,1}}{\|u_{\varepsilon,1}-\tilde{u}_{\varepsilon,1}\|_{L^\infty(\mathbb{T}^2)}} \quad  \text{and} \quad B_\varepsilon\equiv\frac{u_{\varepsilon,2}-\tilde{u}_{\varepsilon,2}}{\|u_{\varepsilon,1}-\tilde{u}_{\varepsilon,1}\|_{L^\infty(\mathbb{T}^2)}} ,$$
 and $(A_{\varepsilon},B_{\varepsilon})$ satisfies \eqref{linear-sym-1}. We consider the following two possible cases.\\
\\
\textit{Case 1.} $\lim_{\varepsilon\to0}\frac{|x_\varepsilon-p|}{\varepsilon}<+\infty$:\\

In this case, there exists $x_0\in\mathbb{R}^2$ such that  $\lim_{\varepsilon\to0}\frac{x_\varepsilon-p}{\varepsilon}=x_0$.
Let   $\hat{A}_\varepsilon(y)\equiv A_\varepsilon(\varepsilon y+p)$ and $\hat{B}_\varepsilon(y)\equiv B(\varepsilon y+p)$.
 Then $(\hat{A}_\varepsilon,\hat{B}_\varepsilon)$ satisfies
\begin{equation*}
\begin{aligned}
\left \{
\begin{array}{ll}
\Delta \hat{A}_\varepsilon- e^{\bar{u}_{\varepsilon,2}+\hat{\eta}_{\varepsilon,1}}\hat{A}_\varepsilon+ e^{\hat{\eta}_{\varepsilon,2}}(1-e^{\hat{u}_{\varepsilon,1}})\hat{B}_\varepsilon=0\quad\mbox{on }~ B_{\frac{r_0}{\varepsilon}}(0),
\\
\\
\Delta \hat{B}_\varepsilon-e^{\bar{u}_{\varepsilon,1}+\hat{\eta}_{\varepsilon,2}}\hat{B}_\varepsilon+ e^{\hat{\eta}_{\varepsilon,1}}(1-e^{\hat{u}_{\varepsilon,2}})\hat{A}_\varepsilon=0\quad\mbox{on }~ B_{\frac{r_0}{\varepsilon}}(0),
\end{array}\right.
\end{aligned}
\end{equation*}
where $\hat{\eta}_{\varepsilon,i}$ is   between $\hat{u}_{\varepsilon,i}$ and $\bar{u}_{\varepsilon,i}$, $i=1,2$.
Then we obtain a subsequence $(\hat{A}_{\varepsilon},\hat{B}_\varepsilon)$ (still denoted in the same way) such that
\begin{equation*}
(\hat{A}_{\varepsilon},\hat{B}_\varepsilon)\to (\hat{A},\hat{B})\ \textrm{uniformly in}\ C^2_{\textrm{loc}}(\mathbb{R}^2)\times C^2_{\textrm{loc}}(\mathbb{R}^2),
\end{equation*}
where $(\hat{A},\hat{B})$ is a bounded solution of
 \begin{equation*}
\begin{aligned}
\left \{
\begin{array}{ll}
\Delta \hat{A}- e^{u_1+u_2}\hat{A}+ e^{u_2}(1-e^{u_1})\hat{B}=0\quad\mbox{on }~ \mathbb{R}^2,
\\ \\
\Delta \hat{B}-e^{u_1+u_2}\hat{B}+ e^{u_1}(1-e^{u_2})\hat{A}=0\quad\mbox{on }~ \mathbb{R}^2.
\end{array}\right.
\end{aligned}
\end{equation*}
By Theorem B,  we obtain $(\hat{A},\hat{B})\equiv (0,0)$.
However, we see that
\begin{equation*}
\begin{aligned}(0,0)&=(\hat{A}(x_0),\hat{B}(x_0))
\\&=\lim_{\varepsilon\to0} \Big(\hat{A}_\varepsilon\Big(\frac{x_\varepsilon-p}{\varepsilon}\Big),\hat{B}_\varepsilon\Big(\frac{x_\varepsilon-p}{\varepsilon}\Big)\Big)
\\&=\lim_{\varepsilon\to0} (A_\varepsilon(x_\varepsilon),B_\varepsilon(x_\varepsilon)),
\end{aligned}
\end{equation*}
where $|A_\varepsilon(x_\varepsilon)|=\frac{|u_{\varepsilon,1}(x_\varepsilon)-\tilde{u}_{\varepsilon,1}(x_\varepsilon)|}{\|u_{\varepsilon,1}-\tilde{u}_{\varepsilon,1}\|_{L^\infty(\mathbb{T}^2)}}=1$ from the choice of $x_\varepsilon$. It is    a contradiction.\\
\\
\textit{Case 2.}  $\lim_{\varepsilon\to0}\frac{|x_\varepsilon-p|}{\varepsilon}=+\infty$:\\

Let
 \begin{equation*}
\begin{aligned}
\left \{
\begin{array}{ll}
\hat{\hat{u}}_{\varepsilon,i}(y)=u_{\varepsilon ,i}(\varepsilon  y+x_\varepsilon )-2\nu_i\ln|\varepsilon  y+ x_\varepsilon -p|+2\nu_i\ln|x_\varepsilon -p|\quad\mbox{on }~ B_{\frac{|x_\varepsilon -p|}{2\varepsilon }}(0),
\\ \\
\bar{\bar{u}}_{\varepsilon,i}(y)=\tilde{u}_{\varepsilon ,i}(\varepsilon  y+x_\varepsilon )-2\nu_i\ln|\varepsilon  y+ x_\varepsilon -p|+2\nu_i\ln|x_\varepsilon -p|\quad\mbox{on }~ B_{\frac{|x_\varepsilon -p|}{2\varepsilon }}(0),
\end{array}\right.
\end{aligned}
\end{equation*}
where  $\nu_i=0$ if  $p\notin\cup_{j=1}^{d_i}\{p_{j,i}\}$ and $\nu_i=\#\{p_{j,i}|p_{j,i}=p \}$.
Then both $(\hat{\hat{u}}_{\varepsilon,1}, \hat{\hat{u}}_{\varepsilon,2})$  and $(\bar{\bar{u}}_{\varepsilon,1}, \bar{\bar{u}}_{\varepsilon,2})$ are solutions of
\begin{equation*}
\begin{aligned}
\left \{
\begin{array}{ll}
\Delta u_{\varepsilon,1}+ \Big|\frac{\varepsilon  y+ x_\varepsilon-p }{|x_\varepsilon-p |}\Big|^{2\nu_2}e^{u_{\varepsilon,2}}(1-\Big|\frac{\varepsilon  y+ x_\varepsilon-p }{|x_\varepsilon-p |}\Big|^{2\nu_1}e^{u_{\varepsilon,1}})=0\quad\mbox{on }~  B_{\frac{|x_\varepsilon-p |}{2\varepsilon }}(0),
\\
\\
\Delta u_{\varepsilon,2}+ \Big|\frac{\varepsilon  y+ x_\varepsilon-p }{|x_\varepsilon-p |}\Big|^{2\nu_1}e^{u_{\varepsilon,1}}(1-\Big|\frac{\varepsilon  y+ x_\varepsilon-p }{|x_\varepsilon-p |}\Big|^{2\nu_2}e^{u_{\varepsilon,2}})=0\quad\mbox{on }~  B_{\frac{|x_\varepsilon-p|}{2\varepsilon }}(0).
\end{array}\right.
\end{aligned}
\end{equation*}
Then the previous arguments, we see that $\lim_{\varepsilon\to0}\sum_{i=1}^2\Big(\sup_{B_{\frac{r_0}{\varepsilon}}(0)}(|\hat{\hat{u}}_{\varepsilon,i}-u_i|+|\bar{\bar{u}}_{\varepsilon,i}-u_i|)\Big)=0$, where $(u_1,u_2)$ is  a unique topological solution of
\begin{equation*}
\Delta u_i+ e^{u_j}(1-e^{u_i})=0 \ \textrm{in}\ \mathbb{R}^2,\  1\le j\neq i\le2.
\end{equation*}
Let   $\hat{\hat{A}}_\varepsilon(y)\equiv A_\varepsilon(\varepsilon  y+x_\varepsilon )$ and $\hat{\hat{B}}_\varepsilon(y)\equiv B(\varepsilon  y+x_\varepsilon )$.
 Then  on $B_{\frac{|x_\varepsilon -p|}{2\varepsilon }}(0)$, $(\hat{\hat{A}}_\varepsilon,\hat{\hat{B}}_\varepsilon)$ satisfies
\begin{equation*}
\begin{aligned}
\left \{
\begin{array}{ll}
\Delta \hat{\hat{A}}_\varepsilon- \Big|\frac{\varepsilon  y+ x_\varepsilon-p }{|x_\varepsilon-p |}\Big|^{2(\nu_1+\nu_2)}e^{\bar{\bar{u}}_{\varepsilon,2}+ \hat{\hat{\eta}}_{\varepsilon,1}}\hat{\hat{A}}_\varepsilon+ \Big|\frac{\varepsilon  y+ x_\varepsilon-p }{|x_\varepsilon-p |}\Big|^{2\nu_2}e^{\hat{\hat{\eta}}_{\varepsilon,2}}(1- \Big|\frac{\varepsilon  y+ x_\varepsilon-p }{|x_\varepsilon-p |}\Big|^{2\nu_1}e^{\hat{\hat{u}}_{\varepsilon,1}})\hat{\hat{B}}_\varepsilon=0,
\\
\\
\Delta \hat{\hat{B}}_\varepsilon- \Big|\frac{\varepsilon  y+ x_\varepsilon-p }{|x_\varepsilon-p |}\Big|^{2(\nu_1+\nu_2)}e^{\bar{\bar{u}}_{\varepsilon,1}+ \hat{\hat{\eta}}_{\varepsilon,2}}\hat{\hat{B}}_\varepsilon+\Big|\frac{\varepsilon  y+ x_\varepsilon-p }{|x_\varepsilon-p |}\Big|^{2\nu_1} e^{\hat{\hat{\eta}}_{\varepsilon,1}}(1- \Big|\frac{\varepsilon  y+ x_\varepsilon-p }{|x_\varepsilon-p |}\Big|^{2\nu_2}e^{\hat{\hat{u}}_{\varepsilon,2}})\hat{\hat{A}}_\varepsilon=0,
\end{array}\right.
\end{aligned}
\end{equation*}
where $\hat{\hat{\eta}}_{\varepsilon,i}$ is   between $\hat{\hat{u}}_{\varepsilon,i}$ and $\bar{\bar{u}}_{\varepsilon,i}$.
Then we obtain a subsequence $(\hat{\hat{A}}_{\varepsilon},\hat{\hat{B}}_\varepsilon)$ (still denoted in the same way) such that
\begin{equation*}
(\hat{\hat{A}}_{\varepsilon},\hat{\hat{B}}_\varepsilon)\to (\hat{\hat{A}},\hat{\hat{B}})\ \textrm{uniformly in}\ C^2_{\textrm{loc}}(\mathbb{R}^2)\times C^2_{\textrm{loc}}(\mathbb{R}^2),
\end{equation*}
where $(\hat{\hat{A}},\hat{\hat{B}})$ is a bounded solution of
 \begin{equation*}
\begin{aligned}
\left \{
\begin{array}{ll}
\Delta \hat{\hat{A}}- e^{u_1+u_2}\hat{\hat{A}}+ e^{u_2}(1-e^{u_1})\hat{\hat{B}}=0\quad\mbox{on }~ \mathbb{R}^2,
\\ \\
\Delta \hat{\hat{B}}-e^{u_1+u_2}\hat{\hat{B}}+ e^{u_1}(1-e^{u_2})\hat{\hat{A}}=0\quad\mbox{on }~ \mathbb{R}^2.
\end{array}\right.
\end{aligned}
\end{equation*}
Then Theorem B implies $(\hat{\hat{A}},\hat{\hat{B}})=(0,0)$.
However, we see that
\begin{equation*}
\begin{aligned}(0,0)&=(\hat{\hat{A}}(0),\hat{\hat{B}}(0))
\\&=\lim_{\varepsilon\to0} (\hat{\hat{A}}_\varepsilon(0),\hat{\hat{B}}_\varepsilon(0))
\\&=\lim_{\varepsilon\to0} (A_\varepsilon(x_\varepsilon),B_\varepsilon(x_\varepsilon)),
\end{aligned}
\end{equation*}
where $|A_\varepsilon(x_\varepsilon)|=\frac{|u_{\varepsilon,1}(x_\varepsilon)-\tilde{u}_{\varepsilon,1}(x_\varepsilon)|}{\|u_{\varepsilon,1}-\tilde{u}_{\varepsilon,1}\|_{L^\infty(\mathbb{T}^2)}}=1$ from the choice of $x_\varepsilon$, and we get a contradiction. So  Theorem \ref{uniqueness} is proved.\hfill $\square$

\textbf{Proof of Theorem \ref{entire}}
 For the proof of Theorem \ref{entire}, the main part of difference from the proof of Theorem \ref{uniqueness} is that we  need to consider the case:  maximum point $x_\varepsilon$ of $|u_{\varepsilon,i}-\tilde{u}_{\varepsilon,i}|$  diverge to $\infty$. In this case,  in view of Theorem A, we have $$u_{\varepsilon,i}(\varepsilon x+x_\varepsilon)\to0\quad\textrm{in}\quad C^2_{\textrm{loc}}(B_{\frac{r}{\varepsilon}}(0)).$$
 Moreover, by using Lemma \ref{linfty}, we see that
\begin{equation*}
(\tilde{A}_\varepsilon(y),\tilde{B}_\varepsilon):=(A_\varepsilon(\varepsilon y+x_\varepsilon), B_\varepsilon(\varepsilon y+x_\varepsilon))\to (\tilde{A},\tilde{B})\ \textrm{uniformly in}\ C^2_{\textrm{loc}}(\mathbb{R}^2)\times C^2_{\textrm{loc}}(\mathbb{R}^2),
\end{equation*}
where $(\tilde{A},\tilde{B})$ is a bounded solution of
 \begin{equation*}
\begin{aligned}
\left \{
\begin{array}{ll}
\Delta \tilde{A}-  \tilde{A}=0\quad\mbox{on }~ \mathbb{R}^2,
\\ \\
\Delta \tilde{B}- \tilde{B}=0\quad\mbox{on }~ \mathbb{R}^2.
\end{array}\right.
\end{aligned}
\end{equation*}
Then Theorem B implies $(\tilde{A},\tilde{B})=(0,0)$
which contradicts $1=\lim_{\varepsilon\to0}|A_\varepsilon(x_\varepsilon)|=|\tilde{A}(0)|$. Now we also complete the proof of Theorem \ref{entire}.\hfill $\square$
\end{section}

%
%
%
%
%
%
%
%
%
%
%
%
%
%

\bibliography{nsf}

\begin{thebibliography}{10}

\bibitem{Bo}
E.~B. Bogomolny.
\newblock The stability of classical solutions.
\newblock {\em Jadernaja Fiz.}, 24(4):861--870, 1976.

\bibitem{BS}
J.~Busca and B.~Sirakov.
\newblock Symmetry results for semilinear elliptic systems in the whole space.
\newblock {\em J. Differential Equations}, 163(1):41--56, 2000.

\bibitem{CY1}
L.~Caffarelli and Y.~Yang.
\newblock Vortex condensation in the {C}hern-{S}imons {H}iggs model: an
  existence theorem.
\newblock {\em Comm. Math. Phys.}, 168(2):321--336, 1995.

\bibitem{CFL}
H.~Chan, C.~Fu, and C.-S. Lin.
\newblock Non-topological multi-vortex solutions to the self-dual
  {C}hern-{S}imons-{H}iggs equation.
\newblock {\em Comm. Math. Phys.}, 231(2):189--221, 2002.

\bibitem{CCL}
J.~Chern, Z.~Chen, and C.S. Lin.
\newblock Uniqueness of topological solutions and the structure of solutions
  for the {C}hern-{S}imons system with two {H}iggs particles.
\newblock {\em Comm. Math. Phys.}, 296(2):323--351, 2010.

\bibitem{C}
K.~Choe.
\newblock Uniqueness of the topological multivortex solution in the self-dual
  {C}hern-{S}imons theory.
\newblock {\em J. Math. Phys.}, 46(1):012305, 22, 2005.

\bibitem{choe}
K.~Choe.
\newblock Multiple existence results for the self-dual {C}hern-{S}imons-{H}iggs
  vortex equation.
\newblock {\em Comm. Partial Differential Equations}, 34(10-12):1465--1507,
  2009.

\bibitem{CKL}
K.~Choe, N.~Kim, and C.S. Lin.
\newblock Existence of self-dual non-topological solutions in the
  {C}hern-{S}imons {H}iggs model.
\newblock {\em Ann. Inst. H. Poincar\'e Anal. Non Lin\'eaire}, 28(6):837--852,
  2011.

\bibitem{Dunne1}
G.~V. Dunne.
\newblock Aspects of {C}hern-{S}imons theory.
\newblock In {\em Aspects topologiques de la physique en basse
  dimension/{T}opological aspects of low dimensional systems ({L}es {H}ouches,
  1998)}, pages 177--263. EDP Sci., Les Ulis, 1999.

\bibitem{dz94}
J.~Dziarmaga.
\newblock Low energy dynamics of ${[\mathrm{U}(1)]}^{N}$ chern-simons solitons.
\newblock {\em Phys. Rev. D}, 49:5469--5479, May 1994.

\bibitem{GT}
D.~Gilbarg and N.~Trudinger.
\newblock {\em Elliptic partial differential equations of second order}.
\newblock Classics in Mathematics. Springer-Verlag, Berlin, 2001.
\newblock Reprint of the 1998 edition.

\bibitem{Gu1}
S.~B. Gudnason.
\newblock {N}on-{A}belian {C}hern--{S}imons vortices with generic gauge groups.
\newblock {\em Nucl. Phys. B}, pages 151--169, 2009,.

\bibitem{Gu2}
S.~B. Gudnason.
\newblock Fractional and semi-local non-{A}belian {C}hern--{S}imons vortices.
\newblock {\em Nucl. Phys. B}, 840:160--185, 2010,.

\bibitem{Hagen}
C.R. Hagen.
\newblock {Parity conservation in Chern-Simons theories and the anyon
  interpretation}.
\newblock {\em Phys.Rev.Lett.}, 68:3821--3825, 1992.

\bibitem{HHL1}
X.~Han, H.~Huang, and C.-S. Lin.
\newblock Bubbling solutions for a skew-symmetric chern--simons system in a
  torus.
\newblock {\em preprint}.

\bibitem{JW}
R.~Jackiw and E.~Weinberg.
\newblock Self-dual {C}hern-{S}imons vortices.
\newblock {\em Phys. Rev. Lett.}, 64:2234--2237, May 1990.

\bibitem{JT}
A.~Jaffe and C.~Taubes.
\newblock {\em Vortices and monopoles}, volume~2 of {\em Progress in Physics}.
\newblock Birkh\"auser, Boston, Mass., 1980.
\newblock Structure of static gauge theories.

\bibitem{KLKLM}
C.~Kim, C.~Lee, P.~Ko, B.~Lee, and H.~Min.
\newblock Schr\"odinger fields on the plane with ${[\mathrm{U}(1)]}^{N}$
  {C}hern-{S}imons interactions and generalized self-dual solitons.
\newblock {\em Phys. Rev. D}, 48:1821--1840, Aug 1993.

\bibitem{LPY}
C.-S. Lin, A.~Ponce, and Y.~Yang.
\newblock A system of elliptic equations arising in {C}hern-{S}imons field
  theory.
\newblock {\em Journal of Functional Analysis}, 247(2):289 -- 350, 2007.

\bibitem{LP}
C.-S. Lin and J.~Prajapat.
\newblock Vortex condensates for relativistic abelian {C}hern-{S}imons model
  with two {H}iggs scalar fields and two gauge fields on a torus.
\newblock {\em Comm. Math. Phys.}, 288(1):311--347, 2009.

\bibitem{LMMS}
G.~Lozano, D.~Marqu\'es, E.~Moreno, and F.~Schaposnik.
\newblock {Non-{A}belian Chern-Simons vortices.}
\newblock {\em Phys. Lett., B}, 654(1-2):27--34, 2007.

\bibitem{PS}
M.~K. Prasad and C.~Sommerfield.
\newblock Exact classical solution for the 't hooft monopole and the julia-zee
  dyon.
\newblock {\em Phys. Rev. Lett.}, 35:760--762, Sep 1975.

\bibitem{SFEG}
S.~Spielman, K.~Fesler, C.~B. Eom, T.~H. Geballe, M.~M. Fejer, and
  A.~Kapitulnik.
\newblock Test for nonreciprocal circular birefringence in {YB}a2{C}u3{O}7 thin
  films as evidence for broken time-reversal symmetry.
\newblock {\em Phys. Rev. Lett.}, 65:123--126, Jul 1990.

\bibitem{SY}
J.~Spruck and Y.~Yang.
\newblock The existence of nontopological solitons in the self-dual
  {C}hern-{S}imons theory.
\newblock {\em Comm. Math. Phys.}, 149(2):361--376, 1992.

\bibitem{T1}
G.~Tarantello.
\newblock Multiple condensate solutions for the {C}hern-{S}imons-{H}iggs
  theory.
\newblock {\em J. Math. Phys.}, 37(8):3769--3796, 1996.

\bibitem{T}
G.~Tarantello.
\newblock Uniqueness of selfdual periodic {C}hern-{S}imons vortices of
  topological-type.
\newblock {\em Calc. Var. Partial Differential Equations}, 29(2):191--217,
  2007.

\bibitem{Wil}
F.~Wilczek.
\newblock Disassembling anyons.
\newblock {\em Phys. Rev. Lett.}, 69:132--135, Jul 1992.

\bibitem{Yan3}
E.~Yanagida.
\newblock Reaction-diffusion systems with skew-gradient structure.
\newblock In {\em International {C}onference on {D}ifferential {E}quations,
  {V}ol. 1, 2 ({B}erlin, 1999)}, pages 760--765. World Sci. Publ., River Edge,
  NJ, 2000.

\bibitem{Yan1}
E.~Yanagida.
\newblock Mini-maximizers for reaction-diffusion systems with skew-gradient
  structure.
\newblock {\em J. Differential Equations}, 179(1):311--335, 2002.

\bibitem{Yan2}
E.~Yanagida.
\newblock Stability analysis for shadow systems with gradient/skew-gradient
  structure.
\newblock {\em S\=urikaisekikenky\=usho K\=oky\=uroku}, (1249):133--142, 2002.
\newblock International Conference on Reaction-Diffusion Systems: Theory and
  Applications (Kyoto, 2001).

\bibitem{yangbook}
Y.~Yang.
\newblock {\em Solitons in field theory and nonlinear analysis}.
\newblock Springer Monographs in Mathematics. Springer-Verlag, New York, 2001.

\end{thebibliography}
\bibliographystyle{plain}

 \end{document}